\documentclass[11pt]{article}

\usepackage{amsmath}
\usepackage{amsthm}
\usepackage{amssymb}
\usepackage{amscd}
\usepackage{pdfsync}
\usepackage{epsfig}
\usepackage{graphicx,color} 
\usepackage{url}
 
 \newtheorem{theorem}{\sc Theorem}[section]
 \newtheorem{lemma}[theorem]{\sc Lemma}

\newtheorem{definition}[theorem]{\sc Definition}

\def\S{\mathbb{S}}

\def\F{\mathbb{F}}

\def\V{\bigotimes_{i\in I} V_i}
\def\PG{\mathrm{PG}}
\def\GL{\mathrm{GL}}
\def\dim{\mathrm{dim}}
\def\rk{\mathrm{rk}}

\def\sh{\mathrm{Sh}}
\def\L{\mathcal{L}}
\def\id{\mathrm{id}}
\begin{document}

\title{Orbits of the stabiliser group of the Segre variety product of three projective lines.}
\author{Michel Lavrauw and John Sheekey}
\date{}
\maketitle
\begin{abstract}
We prove that the stabiliser group $G_X$ of the Segre variety product in $\PG(V)$ of three projective lines over a field $\F$ has four orbits on {\it singular} points of $\PG(V)$, and that $G_X$ has five orbits on points of $\PG(V)$ if $\F$ is {\it finite}.

\end{abstract}
\section{Introduction}
In this paper we consider the action of the stabiliser of the Segre variety product of three projective lines, and
study the number of orbits of projective points.
The Segre variety is a projective algebraic variety named after the italian mathematician Corrado Segre, and they
are considered to be classical objects in algebraic geometry. Over finite fields, most of the results are concerned with the Segre variety product of two projective spaces.
Here we address the Segre variety $\S_{2,2,2}(\F)$ product of three projective lines.
In \cite{HaOdSa2012} the authors study invariant notions of $\S_{2,2,2}(\F_2)$, and they show that over the field $\F_2$ of order two there are exactly five orbits. The orbits of the stabiliser group were previously described without proof in \cite{GlGuMaGuBook}.
Here, we prove that this holds for any finite field. Moreover, we show that there are exactly four orbits if the field is algebraically closed, and that for any field there are exactly four orbits on singular points.

We focus on the geometric properties of Segre varieties to prove our results and use the links with tensor products and algebras when necessary. In particular, our techniques rely on the notions of the rank of tensors, nonsingular tensors and finite semifields, and we introduce these notions in Section \ref{sec:prelims}. Our research was motivated by the theory of finite semifields, and the links with nonsingular tensors brought to our attention by \cite{Lavrauw2012} following \cite{Liebler1981}.

We refer to \cite{Knuth1965} for more on semifields and projective planes, and \cite{LaPo2011} for a recent survey on finite semifields and Galois geometry. We also refer the interested reader to the recent relevant work \cite{Coolsaet2012}, where nonsingular $2\! \times\! 2\! \times \! 2\! \times\!  2$ hypercubes are studied.

We conclude the introduction with the summary of our main results, which are proved in Section \ref{sec:orbits} (see Theorem \ref{thm:orbits} and Theorem \ref{thm:finitecase}).

{\em If $G_X$ denotes the setwise stabiliser, inside the projective group ${\mathrm{PGL}}(V)$, $V=\F^2\otimes \F^2\otimes \F^2$, of the Segre variety product of three projective lines defined over a field $\F$, then $G_X$ has exactly four orbits on {\it singular} points of $\PG(V)$, and $G_X$ has exactly five orbits on points of $\PG(V)$ if $\F$ is {\it finite}.}

\section{Preliminaries}\label{sec:prelims}


Consider the tensor product $\bigotimes_{i\in I} V_i$ ($I=\{1, \ldots, r\}$, $r<\infty$), where $V_1, \ldots, V_r$ are finite dimensional vector spaces over some field $\F$, with $\dim V_i=n_i<\infty$. 
The elements $u\in \V$ that can be written as 
$$u=v_1\otimes \ldots \otimes v_r$$
for some $v_1\in V_1,\ldots, v_r \in V_r$, are called the {\it fundamental tensors} or {\it pure tensors}. The set of fundamental tensors generates the vector space $\V$. The minimal number of pure tensors $u_1,u_2,\ldots u_k$ needed to express a tensor $v\in \bigotimes_{i\in I} V_i$ as a linear combination of $u_1,u_2,\ldots u_k$, is called the {\it rank of $v$}, and is denote by $\rk(v)$.

If $u=v_1\otimes \ldots \otimes v_r$ is
a fundamental tensor in $\V$, and $w_i^\vee \in V_i^\vee$, where $V_i^\vee$ denotes the dual space of $V_i$, then we define $w_i^\vee(u)$ as the tensor
\begin{eqnarray}\label{eqn:contraction}
w_i^\vee(u):=w_i^\vee(v_i)(v_1\otimes \ldots \otimes v_{i-1} \otimes v_{i+1} \otimes \ldots \otimes v_r)\in \bigotimes_{{j\in I\setminus\{i\}}}V_j.
\end{eqnarray}
Since the fundamental tensors span $\V$, this definition naturally extends to a definition of $w_i^\vee(v)$ for any 
$v\in \V$. We call
$$
w_i^\vee(v) \in \bigotimes_{{j\in I\setminus\{i\}}} V_j,
$$
the {\it contraction of $v\in \V$ by $w_i^\vee\in V_i^\vee$}. Also, we say that a vector $u\in V_i$ is {\it nonsingular} if $u\neq 0$, and by induction that a tensor $v\in \V$, $|I|\geq 2$, is {\it nonsingular} if for every $i\in I$ and $w_i^\vee\in V_i^\vee$, $w_i^\vee \neq 0$, the contraction $w_i^\vee(v)$ is nonsingular. A tensor $v\in \V$ is called {\it singular} if it is not nonsingular.

\begin{lemma}\label{lem:alg_closed1}
If $\F$ is algebraically closed, then each tensor of $\V$, with $|I|=3$ and $n_1=n_2=n_3$, is singular.
\end{lemma}
\begin{proof}
A general contraction $w_i^\vee \in V_i^\vee$ of a tensor $v \in \V$ can be identified with
$x_1A_1+x_2A_2+\ldots +x_nA_n$ (where $n=n_1=n_2=n_3$), with respect to some bases of
$V_1$, $V_2$ and $V_3$. Here $(x_1,x_2,\ldots,x_n)$ are the coordinates of $w_i^\vee$ with respect
to the dual basis of $V_i$, and $A_1,A_2,\ldots, A_n$ are $n\times n$-matrices, constituting the
hypercube of $v$ with respect to these bases (these are sometimes called {\it slices}, see 
e.g. \cite[Page 446]{GeKaZe1994}). It follows that $v$ is singular
if and only if there exist $x_1,x_2,\ldots,x_n \in \F$ for which the matrix $x_1A_1+x_2A_2+\ldots +x_nA_n$ is singular. This condition is equivalent to the existence of a solution of a polynomial equation in  $(x_1,x_2,\ldots,x_n)$.
\end{proof}


Put $V=\bigotimes_{i\in I} V_i$, ($I=\{1, \ldots, r\}$, $r<\infty$), where $V_1, \ldots, V_r$ are finite dimensional vectorspaces over some field $\F$, with $\dim V_i=n_i<\infty$, and consider the {\it Segre embedding}
$$
\sigma~:~\PG(V_1)\times \PG(V_2)\times \ldots \PG(V_r) ~\rightarrow ~\PG(V)
$$
$$
(\langle v_1\rangle ,\langle v_2\rangle ,\ldots, \langle v_r\rangle )\mapsto 
\langle v_1\otimes v_2 \otimes \ldots \otimes v_r\rangle.
$$
Let $S_{n_1,n_2,\ldots,n_r}(\F)$ denote the Segre variety in $\PG(V)$. The set of $\F$-rational points of the variety $S_{n_1,n_2,\ldots,n_r}(\F)$ in $\PG(V)$, equals the image of the Segre embedding $\sigma$.

We call a point $x=\langle v \rangle$ in $\PG(V)$  {\it singular} (with respect to the Segre variety $S_{n_1,n_2,\ldots,n_r}(\F)$), if $v\in V$ is singular, and {\it nonsingular} (with respect to the Segre variety $S_{n_1,n_2,\ldots,n_r}(\F)$) otherwise.

\begin{lemma}\label{lem:alg_closed2}
If $\F$ is algebraically closed, then each point of $\PG(n^3-1,\F)$ is singular with respect to the Segre variety
$\S_{n,n,n}(\F)$.
\end{lemma}
\begin{proof}
Immediate from \ref{lem:alg_closed1}.
\end{proof}

We define the {\it rank of a point} $x=\langle v\rangle \in \PG(V)$ as the minimal number of points of $S_{n_1,n_2,\ldots,n_r}(\F)$ needed to span a subspace of $\PG(V)$ containing $x$. We denote this rank by $\rk(x)$. This corresponds to the rank of the tensor $v \in V$, i.e. $\rk(\langle v\rangle)=\rk (v)$. It follows that the points of
the Segre variety have rank 1, and that the rank of a point is invariant under the group of the Segre variety.

\section{Counting orbits}\label{sec:orbits}

In this section, we study the Segre variety, which is the image of the product of three projective lines, i.e. $r=3$ and $n_1=n_2=n_3=2$.
Let $X$ denote the set of $\F$-rational points of a Segre variety $S_{2,2,2}(\F)$, and let $G_X$ denote the setwise stabiliser of $X$ inside the projective group ${\mathrm{PGL}}(7,\F)$. We will prove that $G_X$ has four orbits on {\it singular} points of $\PG(V)$, and that $G_X$ has five orbits on points of $\PG(V)$ if $\F$ is {\it finite}.

It is well known that the group $G_X$ is the group of collineations induced by the wreath product $\GL(\F^2)\wr S_3=K\rtimes S_3$, with base $K=\GL(\F^2)\times \GL(\F^2)\times \GL(\F^2)$.

It is well known that each point $y=\langle y_1\otimes y_2\otimes y_2\rangle \in X$ lies on 
three lines that are contained in $X$: the line $L_1(y)$, which is the image of  $\PG(V_1) \times \langle y_2 \rangle \times \langle y_3\rangle $ under $\sigma$; the line $L_2(y)$, which is
the image of  $\langle y_1\rangle \times \PG(V_2) \times \langle y_3\rangle$ under $\sigma$; and the line $L_3(y)$ which is the image of  $\langle y_1\rangle \times \langle y_2\rangle \times \PG(V_3)$ under $\sigma$. Also, the image
of $\PG(V_1) \times \PG(V_2) \times \langle y_3\rangle $ under $\sigma$ is a hyperbolic quadric $Q_3(y)$, contained in $X$, and containing the point $y$ and the two lines $L_1(y)$ and $L_2(y)$. This also holds for the two other pairs of lines, i.e.
each pair of these lines $(L_i(y),L_j(y))$, $i\neq j$, induces a hyperbolic quadric $Q_k(y)\subset X$, $\{i,j,k\}= \{1,2,3\}$.
Each $Q_k(y)$ spans a solid, which we denote by ${\mathcal{L}}_k(y)$, and $Q_k(y)$ is the intersection of 
${\mathcal{L}}_k(y)$ with $X$.
The union of these solids will be called the {\it shamrock of the point $y$}, and is denoted by $\sh(y)$. The solids $\L_i(y)$ are called the {\it leafs} of the shamrock $\sh(y)$. Clearly any point in the shamrock of a point $y$ has rank at most two, as each point of a leaf $\L_i(y)$ lies on a secant to the hyperbolic quadric $Q_i(y)$.

We introduce the following definition for our convenience.

\begin{definition}
We define the {\it type of a line} $L=\langle y,z\rangle$, spanned by points $y,z\in X$ (say $y=y_1\otimes y_2\otimes y_3$ and $z=z_1\otimes z_2\otimes z_3$) as $(a_1,a_2,a_3)$, where $a_i=\dim\langle y_i,z_i\rangle$. Similarly we will say that a plane $\pi=\langle y,z,w\rangle$, spanned by points $y,z,w\in X$, with
$y=y_1\otimes y_2\otimes y_3$, $z=z_1\otimes z_2\otimes z_3$, and $w=w_1\otimes w_2\otimes w_3$, has type $(a_1,a_2,a_3)$, where $a_i=|\{\langle y_i\rangle,\langle z_i\rangle ,\langle w_i \rangle\}|$.
\end{definition}
The following is obvious.
\begin{lemma}\label{lem:type}
(i) The type of a line (resp. plane) is invariant under the action of the group of collineations induced by the base group
$K=\GL(\F^2)\times \GL(\F^2)\times \GL(\F^2)$.\\
(ii) Using the action of the symmetric group $S_3$ as a subgroup of $G_X$, each $G_X$-orbit of lines (resp. planes) spanned by points of $X$ is represented by a line (resp. plane) of type $(a_1,a_2,a_3)$ with $1\leq a_1\leq a_2\leq a_3\leq2$ (resp. $1\leq a_1\leq a_2\leq a_3\leq3$). 
\end{lemma}

\begin{lemma}\label{lem:singular}
A point of $\PG(7,\F)$ is singular if and only if it is contained in a plane of type $(a_1,a_2,a_3)$, with at least one $a_i\leq 2$.
\end{lemma}
\begin{proof}
Let $x$ be a singular point in $\PG(7,\F)$. 
By \cite[Theorem 4.14]{Lavrauw2012}, the proof of which extends trivially to arbitrary fields, there exists
a point $x_1\in X$ and a Segre variety $Q_k(w)$, $w\in X$, properly contained in $S_{2,2,2}(F)$, such that
$x\subset \langle x_1,Q_k(w)\rangle$. 
Choose $x'\subseteq \langle x,x_1\rangle \cap \langle Q_k(w)\rangle$, and choose a 2-secant 
$L=\langle y',z'\rangle$ through $x'$, where $y',z'\in X$. It follows that the point $x$ is contained in the 
plane $\langle x_1,y',z'\rangle$, which is of type $(a_1,a_2,a_3)$, with at least one $a_k\leq 2$.
\end{proof}

We are now ready to prove our main theorem.

\begin{theorem}\label{thm:orbits}
The group $G_X$ has exactly four orbits on the singular points of $\PG(7,\F)$.\\
\end{theorem}
\begin{proof}
The group $G_X$ acts transitively on the points of $X$, comprising one orbit of points of rank one: $O_1=X$. 

Let $X_2$ be the set of rank two points of $X$. A point $x\in X_2$ lies on at least one $2$-secant $L$ to $X$, say $L=\langle y,z\rangle$, with $y,z\in X$, and by Lemma \ref{lem:type} we may assume that $L$ has type $(1,1,2)$, $(1,2,2)$, or $(2,2,2)$. In the first case, $L$ is contained in $X$, contradicting $x\in X_2$; in the second case the line
$L$ is contained in the leaf $\L_1(y)$ ($z\in Q_1(y)$), and in the last case the line $L$ intersects the shamrock $\sh(y)$ in the point $y$ ($z\notin \sh(y)$).
Let $O_2$ be the set of points of $X_2$ that belong to the shamrock of a point of $X$, and $O_3:=X_2\setminus O_2$.
It follows that $x_2^{G_X} \cap x_3^{G_X}=\emptyset$ for $x_i\in O_i$.

Consider two points $x, x' \in O_2$, say
$x$ is contained in the line spanned by $y,z \in X$ and $x'$ is contained in the line spanned by $y',z'\in X$. Then
$x=y+az$ and $x'=y'+a'z'$, for some $a\neq 0\neq a'$. Since we are interested in orbits under the group $G_X$, by Lemma \ref{lem:type} and part (i), we may assume that $y=y'$, and $z,z' \in Q_1(y)$, and so the lines $\langle y,z\rangle$ and $\langle y,z'\rangle$ are of type $(1,2,2)$. But then one easily verifies that the collineation induced by $g_1\otimes g_2 \otimes g_3$, with $g_1=\id$, $g_2=(y_2\mapsto y_2, z_2\mapsto z'_2)$, and $g_3=(y_3\mapsto y_3, z_3\mapsto (a'/a)z'_3)$ belongs to $G_X$, and maps $x$ to $x'$. It follows that the set $O_2$ forms an orbit of $G_X$.

Similarly, one verifies that also $O_3$, i.e. the set of points of rank two on lines of type $(2,2,2)$, forms an orbit of $G_X$.

We conclude that $O_2$ and $O_3$ are the two orbits of $G_X$ of points of rank two.

Let $X_3$ denote the set of points of rank three, and suppose $x\in X_3$ is a singular point. Since $x\in X_3$, it lies on at least one plane $\pi$ intersecting $X$ in three points, say $\pi=\langle y,z,w\rangle$ with $y,z,w\in X$, and by Lemma \ref{lem:type} and Lemma \ref{lem:singular}, we may assume that $\pi$ has type $(a_1,a_2,a_3)$, with $1\leq a_1\leq a_2\leq a_3\leq3$, and $a_1\leq 2$. If $a_1= 1$, then $\pi$ is contained in a leaf of a point of $X$, a contradiction, since each point in a leaf has rank at most 2. So $\pi$ is of type $(a_1,a_2,a_3)$, with $2= a_1\leq a_2\leq a_3\leq3$.

Suppose $\pi$ has type $(2,2,3)$. Note that if one of the lines determined by two points of $y,z,w$ would have type $(1,1,2)$, then this line would be contained in $X$, and each point of $\pi$ would have rank at most two, a contradiction. So, using Lemma \ref{lem:type}, without loss of generality, we may assume that $y_1=z_1$ and $w_2=z_2$.
This implies that $y\in \L_1(z)$ and $w\in \L_2(z)$. Let $\ell_y$ denote the line 
contained in $Q_1(z)$ through $y$ intersecting $L_3(z)$, $\ell_w$ denote the line
contained in $Q_2(z)$ through $w$ intersecting $L_3(z)$, and note that these lines don't intersect.
It follows that $\pi$ is contained in the three-dimensional space $\langle \ell_y,\ell_w\rangle$. This contradicts $x\in X_3$, since each point of $\langle \ell_y,\ell_w\rangle$ has rank at most two.



Suppose $\pi$ has type $(2,3,3)$. Again, using Lemma \ref{lem:type}, without loss of generality, we may assume that $y_1=z_1$. Denote by $w'$ the unique point of $Q_1(y)$ on the line $L_1(w)$, by $\mathcal C$ the conic section of $Q_1(y)$ with the plane $\langle y,z,w'\rangle$, and by $\ell$ the tangent line at $w'$ to $\mathcal C$.
The planes $\langle x,L_1(w)\rangle$ and $\langle y,z,w'\rangle$ are contained in the three-dimensional space $\langle \pi,w'\rangle$, and intersect in a line $A\ni w'$. If $A\neq \ell$, then $A$ intersects $\mathcal C$ in a second point $w''\neq w'$. But this implies that the plane $\langle x,L_1(w)\rangle$ contains a line $L_1(w)$ and a point $w''\in X\setminus L_1(w)$, and hence only contains points of rank at most two, contradicting $x\in X_3$.
Hence $A=\ell$. Put $\ell_x=\langle w,x\rangle$ and $\alpha=\ell \cap \ell_x$. It follows that $\alpha \neq w'$ and
$\alpha\in \langle L_2(w'),L_3(w')\rangle$, the tangent plane at $w'$ to $Q_1(w')$. Consider any line $r$ 
in $\langle L_2(w'),L_3(w')\rangle$
through $\alpha$ but not through $w$. Then
$r$ intersects $L_i(w')$ in a point $x_i$, $i=1,2$. This implies that $\ell_x$, and hence $x$, is contained in the plane $\langle w,x_1,x_2\rangle$, which is of type $(2,2,2)$.

(Note that we have proved that each line $\ell_x\neq \langle w,w'\rangle$ in the plane $\langle w, \ell \rangle$ is contained in a plane of type $(2,2,2)$.)



So, if $\pi$ has type $(2,3,3)$, then a point of $\pi$ has rank at most two, or is contained in a plane of type $(2,2,2)$.

It follows that a point of $X_3$ is singular if and only if it lies in a plane of type $(2,2,2)$.

Let $O_4$ be the set of points of $X_3$ that belong to a plane of type $(2,2,2)$, and consider two points
$x,x'\in O_4$. Suppose $x\in \pi=\langle y,z,w\rangle$ and $x'\in \pi'=\langle y',z',w'\rangle$, with $y,z,w,y',z',w' \in X$.
By Lemma \ref{lem:type} we may assume that $y_1=z_1$, $z_2=w_2$, $y_3=w_3$ and 
$y'_1=z'_1$, $z'_2=w'_2$, $y'_3=w'_3$.
If $x=ay+bz+w$ and $x'=a'y'+b'z'+w'$ (none of $a,b,a',b'$ is zero, since otherwise $x$ would have rank $\leq 2$), then
the collineation $\tilde{g}\in G_X$ induced by the element $g_1\otimes g_2 \otimes g_3$, with
$g_1=(y_1\mapsto y'_1, w_1\mapsto bb'^{-1}w'_1)$, $g_2=(y_2\mapsto b b'^{-1}y'_2,z_2\mapsto a a'^{-1}z'_2)$, and
$g_3=\id$, maps the point $x$ to
$$
x^{\tilde{g}}=\left \langle (a y_1\otimes y_2\otimes y_3)^{g_1\otimes g_2\otimes g_3}
+(b y_1\otimes z_2\otimes z_3 )^{g_1\otimes g_2\otimes g_3}
+ (w_1\otimes z_2\otimes y_3)^{g_1\otimes g_2\otimes g_3}\right \rangle
$$
$$
=\left \langle \frac{ab}{b'}y'_1\otimes y'_2\otimes y'_3
+\frac{ab}{a'}y'_1\otimes z'_2\otimes z'_3
+\frac{ab}{a'b'} w'_1\otimes z'_2\otimes y'_3
\right \rangle
$$
$$
=\left \langle a'y'_1\otimes y'_2\otimes y'_3
+b'y'_1\otimes z'_2\otimes z'_3
+ w'_1\otimes z'_2\otimes y'_3
\right \rangle=x'.
$$
This shows that $G_X$ acts transitively on $O_4$, and it follows that $O_4$ is the orbit of $G_X$ containing all the
singular points of rank three.
\end{proof}

\begin{theorem}\label{thm:finitecase}
(i) If $\F$ is algebraically closed, then $G_X$ has exactly four orbits on points of $\PG(7,\F)$.\\
(ii) If $\F$ is finite, then $G_X$ has exactly five orbits on the points of $\PG(7,\F)$.
\end{theorem}
\begin{proof}
The first statement is an immediate consequence of Lemma \ref{lem:alg_closed1} and Theorem \ref{thm:orbits}.
The second statement follows from Theorem \ref{thm:orbits}, and the correspondence between
nonsingular tensors and finite semifields (\cite{Liebler1981},\cite{Lavrauw2012}). The orbits of $G_X$ on nonsingular points 
correspond to the isotopism classes of finite semifields (\cite[Theorem 2.1]{Liebler1981} or 
\cite[Theorem 4.3]{Lavrauw2012}). Since each
semifield of dimension two (over its center) is a field (Dickson \cite{Dickson1906}), it follows that
$G_X$ acts transitively on the set of nonsingular points of $\PG(7,\F)$.
\end{proof}

\end{document}